\documentclass[12pt]{article}
\usepackage{amsmath, amsthm, amsfonts, amssymb}
\usepackage{mathrsfs}
\usepackage{bbm}
\usepackage{amssymb}
\usepackage{txfonts}
\usepackage{color}

\newlength{\defbaselineskip}
\newcounter{marnote}

\newcommand{\setlinespacing}[1]%
           {\setlength{\baselineskip}{#1 \defbaselineskip}}

\theoremstyle{plain}
\newtheorem{theorem}{Theorem}[section]
\newtheorem{corollary}[theorem]{Corollary}
\newtheorem{lemma}[theorem]{Lemma}
\newtheorem{prop}[theorem]{Proposition}
\theoremstyle{definition}

\theoremstyle{remark}
\newtheorem{remark}{Remark}[section]
\numberwithin{equation}{section}

\begin{document}


\title{Optimal estimates for the perfect conductivity problem with inclusions close to the boundary}

\author{Haigang
Li\footnote{School of Mathematical Sciences, Beijing Normal University, Laboratory of Mathematics and Complex Systems, Ministry of
   Education, Beijing 100875, China.}~~\footnote{Email: hgli@bnu.edu.cn. }~~
 and ~ Longjuan Xu\footnotemark[1]~~\footnote{Corresponding author. Email: ljxu@mail.bnu.edu.cn }}

\maketitle

\begin{abstract}
When a convex perfectly conducting inclusion is closely spaced to the boundary of the matrix domain, a ``bigger" convex domain containing the inclusion, the electric field can be arbitrary large. We establish both the pointwise upper bound and the lower bound of the gradient estimate for this perfect conductivity problem by using the energy method. These results give the optimal blow-up rates of electric field for conductors with arbitrary shape and in all dimensions. A particular case when a circular inclusion is close to the boundary of a circular matrix domain in dimension two is studied earlier by Ammari, Kang, Lee, Lee and Lim (2007). From the view of methodology, the technique we develop in this paper is significantly different from the previous one restricted to the circular case, which allows us further investigate the general elliptic equations with divergence form.
\end{abstract}

\section{Introduction and main results}

It is well known that in high-contrast fiber-reinforced composites high concentration of extreme electric field or mechanical loads will cause failure initiation in zones, which are created by extreme loads amplified by composite microstructure, including the narrow regions between two adjacent inclusions and the thin gaps between the inclusions and the matrix boundary. The main purpose of this paper is to study the blow-up estimate of $|\nabla{u}|$ where the high concentration of electric field is created. Note that the anti-plane shear model is consistent with the two-dimensional conductivity model. Thus, the blow-up analysis for electric field have a valuable meaning in relation to in the failure analysis of composite material.

There have been many important works on the gradient estimates for the conductivity problem in the presence of inclusions. For two adjacent inclusions $D_{1}$ and $D_{2}$ with $\varepsilon$ apart, Keller \cite{k1} was the first to use analysis to estimate the effective properties of particle reinforced composites. Bonnetier and Vogelius \cite{bv} and Li and Vogelius \cite{lv} proved the uniform boundedness of $|\nabla{u}|$ regardless of $\varepsilon$ provided that the conductivities stay away from $0$ and $\infty$.
Li and Nirenberg \cite{ln} extended the results in \cite{lv} to general divergence form second order elliptic systems including systems of elasticity. This in particular answered in the affirmative the question naturally led to by the numerical indication by Babu\v{s}ka, Andersson, Smith, and Levin \cite{basl} for the boundedness of the strain tensor as $\varepsilon$ tends to 0.
On the other hand, in order to investigate the high-contrast conductivity problem, Ammari, Kang, and Lim \cite{akl} were the first to study the case of the close-to-touching regime of particles whose conductivities degenerate, a lower bound on $|\nabla u|$ was constructed there showing blow-up in both the perfectly conducting and insulating cases. This blow-up was proved to be of order $\varepsilon^{-1/2}$ in $\mathbb R^{2}$. In their subsequent work with H. Lee and J. Lee \cite{aklll} they established upper and lower bounds on the electric field for the close-to touching regime of two circular particles in $\mathbb R^{2}$ with degenerate conductivities. Another interesting case of a particle very close to the boundary is also considered and similar lower and upper bounds for $|\nabla u|$ are established. Subsequently, it has been proved by many mathematicians that for the two close-to-touching inclusions case the generic blow-up rate of $|\nabla{u}|$ blow-up is $\varepsilon^{-1/2}$ in two dimensions, $|\varepsilon\log\varepsilon|^{-1}$ in three dimensions, and $\varepsilon^{-1}$ in dimensions greater than four. See Bao, Li and Yin \cite{bly1,bly2}, as well as Lim and Yun \cite{ly}, Yun \cite{y1,y2}. Further, more detailed, characterizations of the singular behavior of gradient of $u$ have been obtained by Ammari, Ciraolo, Kang, Lee and Yun \cite{ackly}, Ammari, Kang, Lee, Lim and Zribi \cite{akllz}, Bonnetier and Triki \cite{bt1,bt2}, Gorb and Novikov \cite{gn} and Kang, Lim and Yun \cite{kly,kly2}.
For more related work on elliptic equations and systems from composites, see \cite{bll,bll2,bc,dongli,dongzhang,kly0,llby,m,y3} and the references therein. However, for the second case of a particle close to the boundary, to the best of our knowledge, there has not been any further result after \cite{aklll} on the investigation that how the boundary data effects the gradient of the solution until now.

Actually, in \cite{aklll}, Ammari, Kang, Lee,  Lee, and Lim also studied the case that a small disk $B_{r}\subset\mathbb{R}^{2}$, with $\infty$ conductivity, is close to the boundary of a big disk $B_{\rho}$ ($B_{\rho}\supset\,B_{r}$), for which the blow-up rate $\varepsilon^{-1/2}$ is established. Essentially two-dimensional potential theory techniques for circular domain are used in \cite{aklll}, and the authors point out the importance of the three-dimensional case. In this paper we make use of energy method to establish the optimal gradient estimates in all dimensions when a general convex inclusion is very close to the boundary of a ``bigger'' convex domain which contains the inclusion.

Before stating our results, we first describe the nature of our domain. Let $D$ be a bounded open set in $\mathbb{R}^{n}$ $(n\geq2)$, $D_{1}$ be a strictly convex open subset of $D$, both being of class $C^{2,\alpha}$ $(0<\alpha<1)$, and denote the distance $\mathrm{dist}(D_{1},\partial{D})=: \varepsilon>0$.
We further assume that the $C^{2,\alpha}$ norms of
$\partial{D}_{1}$ and  $\partial{D}$ are bounded by some constant independent of
$\varepsilon$.

Suppose that the conductivity of the inclusion $D_{1}$ degenerates to $\infty$; in other words, the inclusion is a perfect conductor. We consider the following conductivity problem\begin{equation}\label{equinfty2}
\begin{cases}
\Delta{u}=0,& \mbox{in}~D\setminus\overline{D}_{1},\\
u=C_{1},&\mbox{on}~\overline{D}_{1},\\
\int_{\partial{D}_{1}}\frac{\partial{u}}{\partial\nu}\big|_{+}=0,&\\
u=\varphi,&\mbox{on}~\partial{D},
\end{cases}
\end{equation}
where $\varphi\in{C}^{2}(\partial D)$, $C_{1}$ is some constant to be determined later, and
$$\frac{\partial{u}}{\partial\nu}\bigg|_{+}:=\lim_{\tau\rightarrow0}\frac{u(x+\nu\tau)-u(x)}{\tau}.$$
Here and throughout this paper $\nu$ is the outward unit normal to
the domain and the subscript $\pm$ indicates the limit from outside
and inside the domain, respectively. The existence, uniqueness and regularity of solutions to equation
\eqref{equinfty2} can be referred to the Appendix in \cite{bly1}, with a minor modification.

Throughout this paper, unless otherwise stated, $C$ denotes a constant, whose value may vary from line to line, depending only on $n,\kappa_{0},\kappa_{1}$ and an upper bound of the $C^{2,\alpha}$ norms of $\partial{D}_{1}$ and $\partial{D}$, but not on $\varepsilon$. Also, we call a constant having such dependence a {\it universal constant}.
Let $P\in\partial{D}$ be the nearest point to $D_{1}$. Let $\overline{PP_{1}}$ denote the shortest line segment between $\partial{D}$ and $\partial{D}_{1}$. Denote
\begin{equation}\label{def_rhon}
\rho_{n}(\varepsilon)=
\begin{cases}
\sqrt{\varepsilon},& n=2,\\
\frac{1}{|\mathrm{ln}\varepsilon|}, & n=3,\\
1, & n\geq4.
\end{cases}
\end{equation}
We have the following gradient estimates in all dimensions.

\begin{theorem}\label{thm1}
Let $D_{1}\subset{D}\subset\mathbb{R}^{n}$ ($n\geq2$) be defined as above. Let
$u\in{H}^{1}(D)\cap{C}^{1}(D\setminus\overline{D}_{1})$
be the solution to \eqref{equinfty2}. Then for $0<\varepsilon<1/2$, we have
\begin{align}\label{upperbound1}
|\nabla{u}(x)|
\leq\frac{C\rho_{n}(\varepsilon)\left|Q[\varphi]\right|}{\varepsilon+\mathrm{dist}^{2}(x,\overline{PP_{1}})}+C
\left(\frac{\mathrm{dist}(x,\overline{PP_{1}})}{\varepsilon+\mathrm{dist}^{2}(x,\overline{PP_{1}})}+1\right)\|\varphi\|_{C^{2}(\partial{D})},\quad\,x\in{D}\setminus{D}_{1},
\end{align}
and if $|Q[\varphi]|\geq\,c^{*}$ for some universal constant $c^{*}>0$, then
\begin{equation}\label{lowerbound1}
|\nabla{u}(x)|\geq
\frac{\rho_{n}(\varepsilon)\left|Q[\varphi]\right|}{C\varepsilon}, \qquad\,x\in\overline{PP_{1}},
\end{equation}
where
\begin{equation}\label{def_Qphi}
Q[\varphi]=\int_{\partial{D}_{1}}\frac{\partial{v}_{0}}{\partial\nu}
\end{equation}
is a bounded functional of $\varphi$, and $v_{0}\in{C}^{2}(D\setminus\overline{D}_{1})$ is uniquely determined by
\begin{equation}\label{equ_v0}
\begin{cases}
\Delta{v}_{0}=0,&\mbox{in}~D\setminus\overline{D}_{1},\\
v_{0}=0,&\mbox{on}~\partial{D}_{1},\\
v_{0}=\varphi(x)-\varphi(P\,),&\mbox{on}~\partial{D}.
\end{cases}
\end{equation}
\end{theorem}

\begin{remark}\label{rem1}
If $\varphi=0$, then the solution of \eqref{equinfty2} is $u\equiv0$. On the other hand, by \eqref{equ_v0}, we have $v_{0}\equiv0$, so $Q[\varphi]=0$. Thus, Theorem \ref{thm1} is obvious. So we only need to prove  it for $\|\varphi\|_{C^{2}(\partial{D})}=1$ by considering $u/\|\varphi\|_{C^{2}(\partial{D})}$. Our result do not really need $D$ and $D_{1}$ to be strictly convex. In fact, our proof of Theorem \ref{thm1} applies to more general situations where $\partial{D}$ and $\partial{D}_{1}$ are relatively strictly convex in a neighborhood of $\overline{PP_{1}}$. Even
when they are not necessarily relatively convex near $P$ and $P_{1}$, while the distance between them remains to be $\varepsilon$, our method also can be applied; for more details, see discussions in Subsection \ref{subsec_general_domain}.
\end{remark}

\begin{remark}\label{rem2}
The upper bound in \eqref{upperbound1} is a pointwise estimate, which provides more information than that in \cite{bly1}. Moreover, the effect of the boundary data to the blow-up of $|\nabla{u}|$ is captured by \eqref{upperbound1} and \eqref{lowerbound1}. In this sense, we can regards \eqref{upperbound1} and \eqref{lowerbound1} as boundary estimates in relation to the interior estimates in \cite{bly1}, where two adjacent inclusions were considered. Furthermore, it turns out that the functional $Q[\varphi]$ plays an important role in the blow-up analysis. It is interesting to know when $|Q[\varphi]|\geq\,c^{*}$ for some positive universal constant $c^{*}$. A sufficient condition for the existence of  $c^{*}$ is given in Subsection \ref{sec_Qphi}.
\end{remark}

The approach developed in the proof of Theorem \ref{thm1} can be extended to study general elliptic equations with a divergence form. Let $n,D_{1},D,\varepsilon$ and $\varphi$ be the same as in Theorem \ref{thm1}, and let
$A_{ij}(x)\in C^{2}(D\setminus\overline{D}_{1})$ be $n\times{n}$ symmetric matrix functions and satisfy the uniform elliptic condition
\begin{equation}\label{ene01a}
\lambda|\xi|^{2}\leq A_{ij}(x)\xi_{i}\xi_{j}\leq\Lambda|\xi|^{2},\quad\forall\xi\in\mathbb R^{n}, \quad x\in D\setminus\overline{D}_{1},
\end{equation}
where $0<\lambda\leq\Lambda<+\infty$. We consider
\begin{equation}\label{equinftydiv}
\begin{cases}
\partial_{i}\left(A_{ij}(x)\partial_{j} u\right)=0,&\mbox{in}~D\setminus\overline{D}_{1},\\
u=C_{1},&\mbox{on}~\overline{D}_{1},\\
\int_{\partial{D}_{1}}(A_{ij}(x)\partial_{j}u)n_{i}\big|_{+}=0,&\\
u=\varphi,&\mbox{on}~\partial{D}.
\end{cases}
\end{equation}
Then

\begin{theorem}\label{thm2}
Let $D_{1}\subset{D}\subset\mathbb{R}^{n}$ ($n\geq2$) be defined as above. Let
$u\in{H}^{1}(D)\cap{C}^{1}(D\setminus\overline{D}_{1})$
be the solution to \eqref{equinftydiv}. Then for $0<\varepsilon<1/2$, we have
\eqref{upperbound1} and \eqref{lowerbound1} hold, where
$$Q[\varphi]=\int_{\partial{D}_{1}}\frac{\partial{V}_{0}}{\partial\nu}
$$
where $V_{0}$ is uniquely determined by
\begin{equation}\label{section3 equ_v0}
\begin{cases}
\partial_{i}\left(A_{ij}(x)\partial_{j}V_{0}\right)=0,&\mbox{in}~D\setminus\overline{D}_{1},\\
V_{0}=0,&\mbox{on}~\partial{D}_{1},\\
V_{0}=\varphi(x)-\varphi(P\,),&\mbox{on}~\partial{D}.
\end{cases}
\end{equation}
\end{theorem}

In order to prove Theorem \ref{thm1}, we decompose the solution $u$ of
\eqref{equinfty2} as follows
\begin{equation}\label{decomposition_u}
u(x)=(C_{1}-\varphi(P\,))v_{1}(x)+v_{0}(x)+\varphi(P\,),\qquad~x\in\,D\setminus\overline{D}_{1},
\end{equation}
where $v_{1}\in{C}^{2}(D\setminus\overline{D}_{1})$ satisfies
\begin{equation}\label{equ_v1}
\begin{cases}
\Delta{v}_{1}=0,&\mbox{in}~D\setminus\overline{D}_{1},\\
v_{1}=1,&\mbox{on}~\partial{D}_{1},\\
v_{1}=0,&\mbox{on}~\partial{D}.
\end{cases}
\end{equation}
Then we have
\begin{equation}\label{decomposition_u2}
\nabla{u}(x)=(C_{1}-\varphi(P\,))\nabla{v}_{1}(x)+\nabla{v}_{0}(x),\qquad~x\in\,D\setminus\overline{D}_{1}.
\end{equation}
Since $u=C_{1}$ on $\partial{D}_{1}$ and $\|u\|_{H^{1}(\Omega)}\leq\,C$ (independent of $\varepsilon$), it follows from the trace embedding theorem that
\begin{equation}\label{C1C2}
|C_{1}|\leq\,C.
\end{equation}
Thus, the proof of Theorem \ref{thm1} is reduced to the estimate of $|\nabla{v}_{1}|$ and $|\nabla v_{0}|$.

Similarly, for Theorem \ref{thm2}, we define $V_{1}\in{C}^{2}(D\setminus\overline{D}_{1})$ satisfying
\begin{equation}\label{section3 equ_v1}
\begin{cases}
\partial_{i}\left(A_{ij}(x)\partial_{j}V_{1}\right)=0,&\mbox{in}~D\setminus\overline{D}_{1},\\
V_{1}=1,&\mbox{on}~\partial{D}_{1},\\
V_{1}=0,&\mbox{on}~\partial{D}.
\end{cases}
\end{equation}

The rest of this paper is organized as follows.  In section \ref{sec_thm1}, we mainly estimate $|\nabla{v}_{1}|$ and $|\nabla v_{0}|$. By constructing an auxiliary function $\bar{u}$, and proving the boundedness of $|\nabla(v_{1}-\bar{u})|$, we show  that $|\nabla\bar{u}|$ is actually the main term of $|\nabla{v}_{1}|$. By the same way, we obtain the estimate of $|\nabla v_{0}|$. Thus, the optimal gradient estimate is  established for convex inclusions in all dimensions. In section \ref{extend}, we give the main ingredients of the proof of Theorem \ref{thm2}. For general elliptic equations with a divergence form, we construct an auxiliary function $\tilde{u}$, associated with the coefficients $A_{ij}$, and then obtain the boundedness of $|\nabla(V_{1}-\tilde{u})|$ and the estimate of $|\nabla V_{0}|$.

\section{Proof of Theorem \ref{thm1}}\label{sec_thm1}

After a possible translation and rotation if necessary, we may assume without loss of generality that $D,D_{1}$ are two strictly convex domains, which satisfy the following:
$$D_{1}\subset\,D,\quad\,P_{1}=\left(0',\varepsilon\right)\in\partial{D}_{1},\quad\,P=\left(0',0\right)\in\partial{D},
\quad\mbox{and}\quad\mathrm{dist}(D_{1},\partial{D})=\varepsilon.$$
Denote $\Omega:=D\setminus\overline{D}_{1}$. Near the origin, we assume that there exists a universal constant $R_{0}$, independent of $\varepsilon$, such that $\partial{D}_{1}$ and $\partial{D}$ can be represented by the graph of
$$x_{n}=\varepsilon+h_{1}(x')\quad\mbox{and}~~x_{n}=h(x'),\quad\mbox{for}~~|x'|\leq\,2R_{0},
$$
respectively, where $h_{1}$ and $h$ satisfy
\begin{equation}\label{h1h0}
\varepsilon+h_{1}(x')>h(x'),\quad\mbox{for}~~|x'|<2R_{0},
\end{equation}
\begin{equation}\label{h1h1}
h_{1}(0')=h(0')=0,\quad\nabla_{x'}h_{1}(0')=\nabla_{x'}h(0')=0,
\end{equation}
\begin{equation}\label{h1h2}
\nabla^{2}_{x'}h_{1}(0')\geq\kappa_{0}I_{n-1}, \quad\nabla^{2}_{x'}h(0')\geq0, \quad\nabla^{2}_{x'}(h_{1}-h)(0')\geq\kappa_{1}I_{n-1},
\end{equation}
where $\kappa_{0},\kappa_{1}>0$, $I_{n-1}$ is the $(n-1)\times(n-1)$ identity matrix, and
\begin{equation}\label{h1h3}
\|h_{1}\|_{C^{2,\alpha}(B'_{2R_{0}})}+\|h\|_{C^{2,\alpha}(B'_{2R_{0}})}\leq{C},
\end{equation}
where $C$ is a universal constant. For $0<r\leq\,2R_{0}$, we denote
$$ \Omega_r:=\left\{(x',x_{n})\in \mathbb{R}^{n}~\big|~h(x')<x_{n}<\varepsilon+h_{1}(x'),~|x'|<r\right\},$$ and
\begin{equation}\label{delta_x}
\delta(x'):=\varepsilon+h_{1}(x')-h(x'),\qquad\forall~(x',x_{n})\in\Omega_{R_{0}}.
\end{equation}

\subsection{Outline of the proof of Theorem \ref{thm1}}

Most of the paper is devoted to these estimates. Now introduce a function $\bar{u}\in{C}^{2}(\mathbb{R}^{n})$, such that $\bar{u}=1$ on
$\partial{D}_{1}$, $\bar{u}=0$ on
$\partial{D}$,
\begin{align}\label{ubar}
\bar{u}(x)
=\frac{x_{n}-h(x')}{\varepsilon+h_{1}(x')-h(x')},\qquad~x\in\,\Omega_{R_{0}},
\end{align}
and
\begin{equation}\label{u_bar_outside}
\|\bar{u}\|_{C^{2}(\Omega\setminus \Omega_{R_{0}})}\leq\,C.
\end{equation}
Using the assumptions on $h_{1}$ and $h$, \eqref{h1h0}--\eqref{h1h3}, a direct calculation gives
\begin{equation}\label{nablau_bar}
\left|\partial_{x_{i}}\bar{u}(x)\right|\leq\frac{C|x'|}{\varepsilon+|x'|^{2}},\quad\,i=1,\cdots,n-1,\quad
\partial_{x_{n}}\bar{u}(x)=\frac{1}{\delta(x')},\quad~x\in\Omega_{R_{0}},
\end{equation}
where $\delta(x')$ is defined by \eqref{delta_x}. Then we have

\begin{prop}\label{prop1}
Assume the above, let $v_{1}, v_{0}\in{H}^1(\Omega)$ be the
weak solution of \eqref{equ_v1} and \eqref{equ_v0}. Then
\begin{equation}\label{nabla_w_i0}
\|\nabla(v_{1}-\bar{u})\|_{L^{\infty}(\Omega)}\leq\,C.
\end{equation}
and
\begin{equation}\label{v1+v2_bounded1}
|\nabla v_{0}(x)|\leq\,C\left(\frac{|x'|}{\varepsilon+|x'|^{2}}+1\right)\|\varphi\|_{C^{2}(\partial{D})},\qquad\,x\in\Omega.
\end{equation}
Consequently,
\begin{equation}\label{nabla_v1}
\frac{1}{C(\varepsilon+|x'|^{2})}\leq|\nabla{v}_{1}(x)|\leq\frac{C}{\varepsilon+|x'|^{2}},\quad\,x\in\Omega_{R_{0}},
\end{equation}
and
$$
|\nabla{v}_{1}(x)|\leq\,C,\quad\quad\ \,x\in \Omega\setminus \Omega_{R_{0}}.
$$
\end{prop}

\begin{remark} Notice that \eqref{v1+v2_bounded1} shows that $|\nabla v_{0}|$ is bounded on the segment $\overline{PP_{1}}$, because of the fact $v_{0}(P_{1})=v_{0}(P\,)=0$. However, for $v_{1}(P_{1})=1$ and $v_{1}(P\,)=0$, $|\nabla{v}_{1}|\sim\varepsilon^{-1}$ on $\overline{PP_{1}}$. Actually, pointwise bound \eqref{nabla_v1} is an improvement of its counterpart in \cite{bly1}, where the maximal principle is the main tool. In order to obtain \eqref{nabla_v1}, we make use of energy method and iteration technology, which is essentially different to that used in \cite{bly1}.
\end{remark}

Proposition \ref{prop1} is the main ingredient in the proof of Theorem \ref{thm1}. The proof will be given in the next subsection.

Define
\begin{equation}\label{def_a_ij}
a_{11}:=\int_{\Omega}|\nabla{v}_{1}|^{2}dx.
\end{equation}
Then using \eqref{nabla_v1},
$$\frac{1}{C}\int_{\Omega}\frac{1}{(\varepsilon+|x'|^{2})^{2}}dx\leq\,a_{11}\leq\,C\int_{\Omega}\frac{1}{(\varepsilon+|x'|^{2})^{2}}dx.$$
By direct integration we obtain the following estimates for $a_{11}$, which is essentially the same as lemmas 2.5--2.7 in \cite{bly1}.
\begin{lemma}\label{lem_bly}(\cite{bly1}) For $n\geq2$,
\begin{align*}
\frac{1}{C\rho_{n}(\varepsilon)}&\leq\,a_{11}\leq\frac{C}{\rho_{n}(\varepsilon)},
\end{align*}
where $\rho_{n}(\varepsilon)$ is defined by \eqref{def_rhon}.
\end{lemma}

\begin{proof}[Proof of Theorem \ref{thm1}.]By the decomposition \eqref{decomposition_u} and the third line of \eqref{equinfty2}, we have
\begin{equation}\label{sysC1C2*}
(C_{1}-\varphi(P\,))\int_{\partial{D}_{1}}\frac{\partial{v}_{1}}{\partial\nu}+\int_{\partial{D}_{1}}\frac{\partial{v}_{0}}{\partial\nu}=0.
\end{equation}
Recalling the definition of $v_{1}$, we have
\begin{equation}\label{aij}
\int_{\partial{D}_{1}}\frac{\partial{v}_{1}}{\partial\nu}=\int_{\partial{D}_{1}}\frac{\partial{v}_{1}}{\partial\nu}v_{1}=-\int_{\Omega}|\nabla{v}_{1}|^{2}=-a_{11}.
\end{equation}
Hence
$$C_{1}-\varphi(P\,)=\frac{Q[\varphi]}{a_{11}}.$$
Thus
$$\nabla{u}=(C_{1}-\varphi(P\,))\nabla{v}_{1}+\nabla{v}_{0}=\frac{Q[\varphi]}{a_{11}}\nabla v_{1}+\nabla v_{0}.$$
Using the upper bounds of $|\nabla{v}_{1}|$ and $|\nabla{v}_{0}|$ in Proposition \ref{prop1} and Lemma \ref{lem_bly}, we obtain \eqref{upperbound1}. If $|Q[\varphi]|\geq\,c^{*}$, then using the lower bound of $|\nabla{v}_{1}|$ in \eqref{nabla_v1} and boundless of $|\nabla v_{0}|$ on the segment $\overline{PP_{1}}$, we have \eqref{lowerbound1} holds on the segment $\overline{PP_{1}}$. Thus, Theorem \ref{thm1} is proved.
\end{proof}

\subsection{Proof of Proposition \ref{prop1}}\label{subsec_prop1}

\begin{proof}
\noindent{\bf STEP 1.} Proof of \eqref{nabla_w_i0}.

Denote
\begin{equation}\label{def_w}
w_{1}:=v_{1}-\bar{u}.
\end{equation}
By the definition of $v_{1}$, \eqref{equ_v1}, and using \eqref{def_w}, we have
\begin{equation}\label{w20}
\left\{
\begin{aligned}
-\Delta{w}_{1}&=\Delta\bar{u}\quad\mbox{in}~\Omega,\\
w_{1}&=0\quad\ \ \mbox{on}~\partial \Omega.
\end{aligned}\right.
\end{equation}
Since
\begin{equation}\label{estimate_baru}
|\bar{u}|+|\nabla\bar{u}|+|\nabla^{2}\bar{u}|\leq\,C,\quad\mbox{in}~~ \Omega\setminus\Omega_{R_{0}/2},
\end{equation}
by the standard elliptic theory, we know that
\begin{equation}\label{nabla_w_out}
|w_{1}|+\left|\nabla{w}_{1}\right|\leq\,C,
\quad\mbox{in}~~ \Omega\setminus\Omega_{R_{0}}.
\end{equation}
Therefore, in order to show \eqref{nabla_w_i0}, we only need to prove
\begin{equation}\label{section2 energyw1}
\left\|\nabla{w}_{1}\right\|_{L^{\infty}(\Omega_{R_{0}})}\leq\,C.
\end{equation}
The rest proof of \eqref{section2 energyw1} is divided into three steps.

{\bf STEP 1.1.} Proof of boundedness of the energy in $\Omega$, that is,
\begin{equation}\label{1energy_w}
\int_{\Omega}\left|\nabla{w}_{1}\right|^{2}\leq\,C.
\end{equation}

Using the maximum principle, we have $0<v_{1}<1$ in $\Omega$, so that
\begin{equation}\label{w_bdd}
\|w_{1}\|_{L^{\infty}(\Omega)}\leq\,C.
\end{equation}
By a direct computation,
\begin{equation}\label{nabla2u_bar}
|\Delta\bar{u}(x)|\leq\frac{C}{\varepsilon+|x'|^{2}},\quad\,x\in\Omega_{R_{0}}.
\end{equation}
Multiplying the equation in \eqref{w20} by $w_{1}$ and integrating by parts, it follows from \eqref{estimate_baru} and \eqref{nabla2u_bar} that
\begin{align*}
\int_{\Omega}|\nabla{w}_{1}|^{2}
=\int_{\Omega}w_{1}\left(\Delta\bar{u}\right)
\leq\,\|w_{1}\|_{L^{\infty}(\Omega)}\left(\int_{\Omega_{R_{0}}}|\Delta\bar{u}|+C\right)
\leq\,C.
\end{align*}
So \eqref{1energy_w} is proved.

{\bf STEP 1.2.} Proof of
\begin{equation}\label{energy_w_inomega_z1}
\int_{\widehat{\Omega}_{\delta}(z')}\left|\nabla{w}_{1}\right|^{2}dx\leq
\begin{cases}C\varepsilon^{n},&\mbox{if}~~|z'|\leq\sqrt{\varepsilon},\\
C|z'|^{2n},&\mbox{if}~~\sqrt{\varepsilon}<|z'|\leq\,R_{0},
\end{cases}
\end{equation}
where
\begin{equation*}
 \widehat{\Omega}_{t}(z'):=\left\{x\in \mathbb{R}^{n}~\big|~h(x')<x_{n}<\varepsilon+h_{1}(x'),~|x'-z'|<{t}\right\}.
\end{equation*}

The iteration scheme here we use is similar in spirit to that used in  \cite{bll,llby}. For $0<t<s<R_{0}$, let $\eta$ be a smooth cutoff function satisfying $\eta(x')=1$ if $|x'-z'|<t$, $\eta(x')=0$ if $|x'-z'|>s$, $0\leq\eta(x')\leq1$ if $t\leq|x'-z'|\leq\,s$, and $|\nabla_{x'}\eta(x')|\leq\frac{2}{s-t}$. Multiplying the equation in \eqref{w20} by $w_{1}\eta^{2}$ and integrating by parts
leads  to
\begin{align}\label{FsFt11}
\int_{\widehat{\Omega}_{t}(z')}|\nabla{w}_{1}|^{2}\leq\,\frac{C}{(s-t)^{2}}\int_{\widehat{\Omega}_{s}(z')}|w_{1}|^{2}
+(s-t)^{2}\int_{\widehat{\Omega}_{s}(z')}\left|\Delta\bar{u}\right|^{2}.
\end{align}

{\bf Case 1.} For $|z'|\leq \sqrt{\varepsilon}$. For $0<s<\sqrt{\epsilon}$, using \eqref{nabla2u_bar}, we have

\begin{equation}\label{integal_Lubar11_in}
\int_{\widehat{\Omega}_{s}(z')}
\left|\Delta\bar{u}\right|^{2}
\leq\frac{Cs^{n-1}}{\varepsilon},\quad\mbox{if}~\,0<s<\sqrt{\varepsilon}.
\end{equation}
Note that
\begin{align}\label{energy_w_square_in}
\int_{\widehat{\Omega}_{s}(z')}|w_{1}|^{2}
\leq\,C\varepsilon^{2}\int_{\widehat{\Omega}_{s}(z')}|\nabla{w}_{1}|^{2},
\quad\mbox{if}~\,0<s<\sqrt{\varepsilon}.
\end{align}
Denote
$$F(t):=\int_{\widehat{\Omega}_{t}(z')}|\nabla{w}_{1}|^{2}.$$
It follows from \eqref{FsFt11}, \eqref{integal_Lubar11_in} and
\eqref{energy_w_square_in} that
\begin{equation}\label{tildeF111_in}
F(t)\leq\,\left(\frac{c_{1}\varepsilon}{s-t}\right)^{2}F(s)+C(s-t)^{2}\frac{s^{n-1}}{\varepsilon},
\quad\forall~0<t<s<\sqrt{\varepsilon},
\end{equation}
where $c_1$ is a universal constant.

Let $k=\left[\frac{1}{4c_{1}\sqrt{\varepsilon}}\right]$ and $t_{i}=\delta+2c_{1}i\varepsilon$, $i=0,1,2,\cdots,k$. Then
by \eqref{tildeF111_in} with $s=t_{i+1}$ and $t=t_{i}$, we have
$$F(t_{i})\leq\,\frac{1}{4}F(t_{i+1})
+C(i+1)^{n-1}\varepsilon^{n}.$$
After $k$ iterations,
using (\ref{1energy_w}), we have
\begin{eqnarray*}
F(t_{0})
\leq (\frac{1}{4})^{k}F(t_{k})
+C\varepsilon^{n}\sum_{i=0}^{k-1}(\frac{1}{4})^{i}(i+1)^{n-1}
\leq\,C\varepsilon^{n}.
\end{eqnarray*}
Therefore
$$\int_{\widehat{\Omega}_{\delta}(z')}|\nabla{w}_{1}|^{2}\leq\,C\varepsilon^{n}.$$

{\bf Case 2.} For $\sqrt{\varepsilon}\leq|z'|\leq\,R_{0}$. Estimate \eqref{integal_Lubar11_in} becomes
\begin{equation}\label{integal_Lubar11}
\int_{\widehat{\Omega}_{s}(z')}\left|\Delta\bar{u}\right|^{2}\leq\frac{Cs^{n-1}}{|z'|^{2}},\quad\mbox{if}~0<s<|z'|.
\end{equation}
Estimate \eqref{energy_w_square_in} becomes
\begin{align}\label{energy_w_square}
\int_{\widehat{\Omega}_{s}(z')}|w_{1}|^{2}
\leq&\,C|z'|^{4}\int_{\widehat{\Omega}_{s}(z')}|\nabla{w}_{1}|^{2}, \quad\mbox{if}~\,0<s<|z'|.
\end{align}
Estimate \eqref{tildeF111_in} becomes, in view of \eqref{FsFt11},
\begin{equation}\label{tildeF111}
F(t)\leq\,\left(\frac{c_{2}|z'|^{2}}{s-t}\right)^{2}F(s)+C(s-t)^{2}\frac{s^{n-1}}{|z'|^{2}},
\quad\forall~0<t<s<|z'|,
\end{equation}
where $c_2$ is another universal constant.

Let $k=\left[\frac{1}{4c_{2}|z'|}\right]$ and $t_{i}=\delta+2c_{2}i\,|z'|^{2}$, $i=0,1,2,\cdots,k$. Then applying \eqref{tildeF111} with $s=t_{i+1}$ and $t=t_{i}$, we have
$$F(t_{i})\leq\,\frac{1}{4}F(t_{i+1})+\frac{C(t_{i+1}-t_{i})^{2}t_{i+1}^{n-1}}{|z'|^{2}}
\leq\,\frac{1}{4}F(t_{i+1})+C(i+1)^{n-1}|z'|^{2n},
$$
After $k$ iterations, using (\ref{1energy_w}) we have
\begin{eqnarray*}
F(t_{0}) \leq (\frac{1}{4})^{k}F(t_{k})+C|z'|^{2n}\sum_{i=0}^{k-1}(\frac{1}{4})^{i}(i+1)^{n-1}
\leq C|z'|^{2n}.
\end{eqnarray*}
This implies that
$$\int_{\widehat{\Omega}_{\delta}(z')}|\nabla{w}_{1}|^{2}\leq\,C|z'|^{2n}.$$
\eqref{energy_w_inomega_z1} is proved.

{\bf STEP 1.3.} Rescaling and $L^{\infty}$ estimates. Denote $\delta:=\delta(z')$. Making a change of variables
\begin{equation}\label{changeofvariant}
 \left\{
  \begin{aligned}
  &x'-z'=\delta y',\\
  &x_n=\delta y_n,
  \end{aligned}
 \right.
\end{equation}
then $\widehat{\Omega}_{\delta}(z')$ becomes $Q'_{1}$, where
$$Q'_{r}=\left\{y\in\mathbb{R}^{n}~\Big|~\frac{1}{\delta}h(\delta{y}'+z')<y_{n}
<\frac{\varepsilon}{\delta}+\frac{1}{\delta}h_{1}(\delta{y}'+z'),~|y'|<r\right\},\quad\mbox{for}~~r\leq1,$$ and the top and
bottom boundaries
become
$$
y_n=\hat{h}_{1}(y'):=\frac{1}{\delta}
\left(\varepsilon+h_{1}(\delta\,y'+z')\right),\quad|y'|<1,$$
and
$$y_n=\hat{h}(y'):=\frac{1}{\delta}h(\delta\,y'+z'), \quad |y'|<1.
$$
 Then
$$\hat{h}_{1}(0')-\hat{h}(0'):=\frac{\varepsilon+h_{1}(z')-h(z')}{\delta}=1,$$
and by \eqref{h1h0}--\eqref{h1h3},
$$|\nabla_{x'}\hat{h}_{1}(0')|+|\nabla_{x'}\hat{h}(0')|\leq\,C|z'|,
\quad|\nabla_{x'}^{2}\hat{h}_{1}(0')|+|\nabla_{x'}^{2}\hat{h}(0')|\leq\,C\delta.$$
Since $R_{0}$ is small, $\|\hat{h}_{1}\|_{C^{1,1}((-1,1)^{n-1})}$ and $\|\hat{h}\|_{C^{1,1}((-1,1)^{n-1})}$ are small and $Q'_{1}$ is essentially a unit square (or a unit cylinder) as far as applications of Sobolev embedding
theorems and classical $L^{p}$ estimates for elliptic equations are concerned.
Let
\begin{equation}\label{def_U}
U(y', y_n):=\bar{u}(\delta\,y'+z',\delta\,y_{n}),\quad\,W(y', y_n):=w_{1}(\delta\,y'+z',\delta\,y_{n}),
\quad\,y\in{Q}'_{1},
\end{equation}
then by the equation in \eqref{w20},
\begin{align}
-\Delta{W}
=\Delta{U},
\quad\quad\,y\in{Q'_{1}}.
\end{align}
where
$$\left|\Delta{U}\right|=\delta^{2}\left|\Delta\bar{u}\right|.$$
Since $W=0$ on the top and
bottom boundaries of $Q'_{1}$, it follows from the
 Poincar\'{e} inequality that
$$\left\|W\right\|_{H^{1}(Q'_{1})}\leq\,C\left\|\nabla{W}\right\|_{L^{2}(Q'_{1})}.$$
By
 $W^{2,p}$ estimates for elliptic equations and
Sobolev embedding theorems,
for $p>n$,
\begin{align*}
\left\|\nabla{W}\right\|_{L^{\infty}(Q'_{1/2})}\leq\,
C\left\|W\right\|_{W^{2,p}(Q'_{1/2})}
\leq\,C\left(\left\|\nabla{W}\right\|_{L^{2}(Q'_{1})}+\left\|\Delta{U}\right\|_{L^{\infty}(Q'_{1})}\right).
\end{align*}
Therefore
\begin{equation}
\left\|\nabla{w_{1}}\right\|_{L^{\infty}(\widehat{\Omega}_{\delta/2}(z'))}\leq\,
\frac{C}{\delta}\left(\delta^{1-\frac{n}{2}}\left\|\nabla{w_{1}}\right\|_{L^{2}(\widehat{\Omega}_{\delta}(z'))}
+\delta^{2}\left\|\Delta\bar{u}\right\|_{L^{\infty}(\widehat{\Omega}_{\delta}(z'))}\right).
\label{AAA}
\end{equation}

{\bf Case 1.} For $|z'|\leq\sqrt{\varepsilon}$.
Using \eqref{nabla2u_bar} and \eqref{delta_x},
$$\delta^{2}\left|\Delta\bar{u}\right|\leq
\frac{C \delta^{2}}{\varepsilon}\leq\,C
\varepsilon, \qquad
\mbox{in}\
\widehat \Omega_\delta(z').$$
It follows from (\ref{AAA}) and (\ref{energy_w_inomega_z1}) that
$$\left|\nabla{w}_{1}(z',x_{n})\right|\leq\frac{C(\delta^{1-\frac{n}{2}}\varepsilon^{\frac{n}{2}}+\varepsilon)}{\delta}
\leq\,C,
\qquad\forall \ h(z')<x_{n}<\varepsilon+h_1(z').$$

{\bf Case 2.} For $\sqrt{\varepsilon}\leq|z'|\leq\,R_{0}$.
Using \eqref{nabla2u_bar} and \eqref{delta_x},
$$\delta^{2}\left|\Delta\bar{u}\right|\leq\frac{C\delta^{2}}{|z'|^{2}}\leq\,C|z'|^{2},
\qquad \mbox{in}\ \widehat\Omega_\delta(z').$$
We deduce from (\ref{AAA}) and (\ref{energy_w_inomega_z1}) that
$$\left|\nabla{w}_{1}(z',x_{n})\right|\leq\frac{C(\delta^{1-\frac{n}{2}}|z'|^{n}+|z'|^{2})}{\delta}\leq\,C,
\qquad\forall \
h(z')<x_{n}<\varepsilon+h_1(z').$$
Estimate \eqref{nabla_w_i0} is established.

\noindent{\bf STEP 2.} Proof of \eqref{v1+v2_bounded1}.

Similar to the proof of \eqref{nabla_w_i0}, we introduce a function $\hat{u}\in C^{2}(\mathbb R^{n})$, such that $\hat{u}=0$ on $\partial{D}_{1}$, $\hat{u}=\varphi(x)-\varphi(0\,)$ on $\partial{D}$,
\begin{align}\label{uhat}
\hat{u}(x)
=(1-\bar{u}(x))\Big(\varphi(x^{\prime},h(x^{\prime}))-\varphi(0)\Big),\quad\mbox{in}~~\Omega_{R_{0}},
\end{align}
and
\begin{equation}\label{nablau_hat_outside}
\|\hat{u}\|_{C^{2}(\mathbb{R}^{n}\setminus \Omega_{R_{0}})}\leq\,C.
\end{equation}
Denote
\begin{equation*}
w_{0}:= v_{0}-\hat{u}.
\end{equation*}
Then by the definitions of $v_{0}$, \eqref{equ_v0},
\begin{equation}\label{w}
\begin{cases}
-\Delta {w}_{0}=\Delta\hat{u},\quad&\mbox{in}~\Omega,\\
 w_{0}=0,\quad\quad&\mbox{on}~\partial\Omega.
\end{cases}
\end{equation}
Similarly as \eqref{nabla_w_out} and \eqref{w_bdd}, we have
\begin{equation}\label{nabla_w}
\| w_{0}\|_{L^{\infty}(\Omega)}+\|\nabla {w}_{0}\|_{L^{\infty}(\Omega\setminus\Omega_{R_{0}/2})}\leq\,C.
\end{equation}
Thus, in order to prove \eqref{v1+v2_bounded1}, we only need to prove
\begin{equation*}
\left\|\nabla {w}_{0}\right\|_{L^{\infty}(\Omega_{R_{0}})}\leq\,C.
\end{equation*}

By a direct calculation, we have for $x\in\Omega_{R_{0}}$,
\begin{equation*}
\begin{aligned}
\partial_{x_{i}}\hat{u}&=\left(1-\bar{u}\right)\Big(\partial_{x_{i}}\varphi+\partial_{x_{n}}\varphi\partial_{x_{i}}h\Big)-\partial_{x_{i}}\bar{u}~\Big(\varphi(x',h(x'))-\varphi(0)\Big),\quad i=1,2,\cdots,n-1,\\
\partial_{x_{n}}\hat{u}&=-\partial_{x_{n}}\bar{u}~\Big(\varphi(x',h(x'))-\varphi(0)\Big).
\end{aligned}
\end{equation*}
Using the assumption on $\varphi$, we have
\begin{equation}\label{1-varphi}
\Big|\varphi(x',h(x'))-\varphi(0)\Big|\leq\,C\|\varphi\|_{C^{1}(\partial{D})}\,|x'|,\qquad\mbox{in}~~\Omega_{R_{0}},
\end{equation}
then in view of \eqref{h1h0}--\eqref{h1h3},
\begin{equation}\label{nabla_uhat}
|\nabla\hat{u}|\leq\,C\left(\frac{|x'|}{\varepsilon+|x^{\prime}|^{2}}+1\right)\|\varphi\|_{C^{1}(\partial{D})},\qquad\mbox{in}~~\Omega_{R_{0}}.
\end{equation}
Futhermore,
\begin{align*}
\Delta\hat{u}=&(1-\bar{u})\Big(\Delta_{x'}\varphi+2\nabla_{x'}(\partial_{x_{n}}\varphi)\cdot\nabla_{x'}h+(\partial_{x_{n}}\varphi)\Delta_{x'}h+\partial_{x_{n}x_{n}}\varphi(\nabla_{x'}h)^{2}\Big)\\
&-2\nabla_{x'}\bar{u}\cdot\Big(\nabla_{x'}\varphi+(\partial_{x_{n}}\varphi)\nabla_{x'}h\Big)-\Delta\bar{u}\,\Big(\varphi(x',h(x'))-\varphi(0)\Big),
\end{align*}
and using \eqref{1-varphi} and \eqref{h1h0}--\eqref{h1h3} again,
$$|\Delta\hat{u}|\leq\,C\left(\frac{|x'|}{\varepsilon+|x^{\prime}|^{2}}+1\right)\|\varphi\|_{C^{2}(\partial{D})},\qquad\mbox{in}~~\Omega_{R_{0}},$$
which is better than \eqref{nabla2u_bar}. Therefore, the rest of the proof is completely the same as step 1.1--1.3 above. \eqref{v1+v2_bounded1} is proved. The proof of Proposition \ref{prop1} is completed.
\end{proof}

\subsection{Estimates of $Q[\varphi]$}\label{sec_Qphi}

In order to identify the lower bound \eqref{lowerbound1}, we estimate $|Q[\varphi]|$ in this subsection. Let $D_{1}^{*}:=D_{1}-(0',\varepsilon)$ and $\Omega^{*}:=D\setminus\overline{D_{1}^{*}}$ and define
\begin{equation}\label{equ_v12*}
\begin{cases}
\Delta{v}_{1}^{*}=0,&\mbox{in}~\Omega^{*},\\
v_{1}^{*}=1,&\mbox{on}~\partial{D}_{1}^{*}\setminus\{0\},\\
v_{1}^{*}=0,&\mbox{on}~\partial{D}\setminus\{0\},
\end{cases}
\qquad\mbox{and}\qquad
\begin{cases}
\Delta{v}_{0}^{*}=0,&\mbox{in}~\Omega^{*},\\
v_{0}^{*}=0,&\mbox{on}~\partial{D}_{1}^{*},\\
v_{0}^{*}=\varphi(x)-\varphi(0\,),&\mbox{on}~\partial{D}.
\end{cases}
\end{equation}

\begin{lemma}\label{lem2.31}
There exists a unique $v_{i}^{*}\in{L}^{\infty}(\Omega^{*})\cap\,C^{0}(\overline{\Omega^{*}}\setminus\{0\})\cap\,C^{2}(\Omega^{*})$, $i=0,1$, which solve \eqref{equ_v12*}. Moreover, $v_{i}^{*}\in{C}^{1}(\overline{\Omega^{*}}\setminus\{0\})$.
\end{lemma}

\begin{proof}
The existence of solutions of \eqref{equ_v12*} can easily be obtained by Perron's method, see theorem 2.12 and lemma 2.13 in \cite{gt}. For the readers' convenience, we give a simple proof of the uniqueness for $n\geq3$. The case $n=2$ is similar. We only need to prove that $0$ is the only solution in $L^{\infty}(\Omega^{*})\cap{C}^{0}(\overline{\Omega^{*}}\setminus\{0\})\cap{C}^{2}(\Omega^{*})$ of the following equation
$$
\begin{cases}
\Delta{w}=0,&\mbox{in}~\Omega^{*},\\
w=0,&\mbox{on}~\partial\Omega^{*}\setminus\{0\}.
\end{cases}
$$
Indeed, noticing that $w=0$ on $\partial\Omega^{*}\setminus\{0\}$, it follows that for any $\epsilon>0$,
$$|w(x)|\leq\frac{\epsilon^{n-2}\|w\|_{L^{\infty}(\Omega^{*})}}{|x|^{n-2}}, \quad\mbox{on}~~\partial(\Omega^{*}\setminus{B}_{\epsilon}(0)).$$
Using the maximum principle, we have
$$|w(x)|\leq\frac{\epsilon^{n-2}\|w\|_{L^{\infty}(\Omega^{*})}}{|x|^{n-2}}, \quad\mbox{in}~~\Omega^{*}\setminus{B}_{\epsilon}(0).$$
Thus, $w=0$ in $\Omega^{*}$. The additional regularity follows from standard elliptic estimates and the smoothness of $\partial{D}_{1}$ and $\partial{D}$.
\end{proof}

\begin{lemma}\label{lem2.3}
For $i=0,1$,
\begin{equation}\label{equ_viconvergence}
v_{i}\rightarrow\,v_{i}^{*},\quad\mbox{in}~~C_{loc}^{2}(\Omega^{*}), \quad\mbox{as}~~\varepsilon\rightarrow0,
\end{equation}
and
\begin{equation}\label{Q_convergence}
\int_{\partial{D}_{1}}\frac{\partial v_{0}}{\partial\nu}\rightarrow\,\int_{\partial{D}_{1}^{*}}\frac{\partial v_{0}^{*}}{\partial\nu}, \quad\mbox{as}~~\varepsilon\rightarrow0.
\end{equation}
\end{lemma}

\begin{proof}
By the maximum principle, $\|v_{i}\|_{L^{\infty}}$ is bounded by a constant independent of $\varepsilon$. By the uniqueness  part of Lemma \ref{lem2.31}, we obtain \eqref{equ_viconvergence} using standard elliptic estimates. It follows from the definition of $v_{1}$ and the Green's formula that
$$\int_{\partial{D}_{1}}\frac{\partial v_{0}}{\partial\nu}=\int_{\partial{D}_{1}}\frac{\partial v_{0}}{\partial\nu}v_{1}=\int_{\partial\Omega}\frac{\partial v_{0}}{\partial\nu}v_{1}=\int_{\partial\Omega}\frac{\partial{v}_{1}}{\partial\nu} v_{0}=\int_{\partial{D}}\frac{\partial{v}_{1}}{\partial\nu}~\Big(\varphi(x)-\varphi(0)\Big).$$
Similarly,
\begin{equation}\label{equ_Q*}
\int_{\partial{D}_{1}^{*}}\frac{\partial v_{0} ^{*}}{\partial\nu}=\int_{\partial{D}_{1}^{*}}\frac{\partial v_{0} ^{*}}{\partial\nu}v_{1}^{*}=\int_{\partial{D}}\frac{\partial{v}_{1}^{*}}{\partial\nu}~\Big(\varphi(x)-\varphi(0)\Big).
\end{equation}
\eqref{Q_convergence} follows from \eqref{equ_viconvergence}.
\end{proof}

Define
$$Q^{*}[\varphi]=\int_{\partial{D}_{1}^{*}}\frac{\partial v_{0} ^{*}}{\partial\nu}.$$
By Lemma \ref{lem2.3},
$$Q[\varphi]\rightarrow\,Q^{*}[\varphi],\qquad\mbox{as}~~\varepsilon\rightarrow0.$$

\begin{corollary}
If $\varphi\in{C}^{2}(\partial{D})$ satisfies $Q^{*}[\varphi]\neq0$, then $|Q[\varphi]|\geq\,c^{*}$, for some positive universal constant $c^{*}$ which is independent of $\varepsilon$.
\end{corollary}

\begin{remark}
It follows from the definition of $Q^{*}[\varphi]$ and \eqref{equ_Q*} that $Q^{*}[\varphi]=\int_{\partial{D}}\frac{\partial{v}_{1}^{*}}{\partial\nu}\big(\varphi(x)-\varphi(0)\big)$.
\end{remark}

\subsection{More general $D_{1}$ and $D$}\label{subsec_general_domain}

As mentioned in Remark \ref{rem1}, the strict convexity assumption of $D_{1}$ and $D$ can be weakened. In fact, our proof of Proposition \ref{prop1} applies , with minor modification, to more general situations: In $\mathbb{R}^{n}$, $n\geq2$, under the same assumptions in the beginning of Section \ref{sec_thm1} except the strict convexity assumptions \eqref{h1h2}. We assume that
\begin{equation}\label{h1h_convex1}
\lambda_{0}|x'|^{m}\leq\,h_{1}(x')-h(x')\leq\lambda_{1}|x'|^{m},\quad\mbox{for}~~|x'|<2R_{0},
\end{equation}
and
\begin{equation}\label{h1h_convex2}
|\nabla_{x'}h_{1}(x')|,|\nabla_{x'}^{2}h(x')|\leq\,C|x'|^{m-1},\quad|\nabla_{x'}^{2}h_{1}(x')|,|\nabla_{x'}^{2}h(x')|\leq\,C|x'|^{m-2},\quad\mbox{for}~~|x'|<2R_{0},
\end{equation}
for some $\varepsilon$-independent constants $0<\lambda_{0}<\lambda_{1}$, and $m\geq2$. Clearly,
$$\frac{1}{C}(\varepsilon+|x'|^{m})\leq\delta(x')\leq\,C(\varepsilon+|x'|^{m}).$$
Then by the same procedure in the proof of Proposition \ref{prop1}, we have
\begin{prop}\label{prop2}
Assume the above, under the assumptions \eqref{h1h_convex1} and \eqref{h1h_convex2}, instead of \eqref{h1h2}. Let $v_{1}, v_{0}\in{H}^1(\Omega)$ be the
weak solution of \eqref{equ_v1} and \eqref{equ_v0}. Then
\begin{equation}\label{nabla_w_i0m}
\|\nabla(v_{1}-\bar{u})\|_{L^{\infty}(\Omega)}\leq\,C.
\end{equation}
and
\begin{equation}\label{v1+v2_bounded1m}
 |\nabla v_{0}(x)|\leq\,C\left(\frac{|x'|^{m-1}}{\varepsilon+|x'|^{m}}+1\right)\|\varphi\|_{C^{2}(\partial{D})},\qquad\,x\in\Omega.
\end{equation}
Consequently,
\begin{equation}\label{nabla_v1m}
\frac{1}{C(\varepsilon+|x'|^{m})}\leq|\nabla{v}_{1}(x)|\leq\frac{C}{\varepsilon+|x'|^{m}},\quad\,x\in\Omega_{R_{0}},
\end{equation}
and
$$
|\nabla{v}_{1}(x)|\leq\,C,\quad\quad\ \,x\in \Omega\setminus \Omega_{R_{0}}.
$$
\end{prop}

Thus, using \eqref{nabla_v1m}, we have
$$\frac{1}{C}\int_{\Omega}\frac{1}{(\varepsilon+|x'|^{m})^{2}}dx\leq\,a_{11}\leq\,C\int_{\Omega}\frac{1}{(\varepsilon+|x'|^{m})^{2}}dx.$$
By direct integration we obtain the following estimates for $a_{11}$, insetad of Lemma \ref{lem_bly}.
\begin{lemma}\label{lem_blym}
For $n\geq2$ and $m\geq2$,
$$
\frac{1}{C\rho_{n}^{m}(\varepsilon)}\leq\,a_{11}\leq\frac{C}{\rho_{n}^{m}(\varepsilon)},\quad\mbox{where}~~\rho_{n}^{m}(\varepsilon)=
\begin{cases}
\varepsilon^{1-\frac{n-1}{m}},&\mbox{if}~~n-1<m,\\
\frac{1}{|\ln\varepsilon|},&\mbox{if}~~n-1=m,\\
1,&\mbox{if}~~n-1>m.
\end{cases}
$$
\end{lemma}
Hence, we have the following more general theorem.
\begin{theorem}\label{thm3}
Let $D_{1}\subset{D}\subset\mathbb{R}^{n}$ ($n\geq2$) be defined as above, under the assumptions \eqref{h1h_convex1} and \eqref{h1h_convex2}, instead of \eqref{h1h2}.  Let
$u\in{H}^{1}(D)\cap{C}^{1}(D\setminus\overline{D}_{1})$
be the solution to \eqref{equinfty2}. Then for $0<\varepsilon<1/2$, we have
\begin{align}\label{upperbound1m}
|\nabla{u}(x)|
\leq\frac{C\rho_{n}^{m}(\varepsilon)\left|Q[\varphi]\right|}{\varepsilon+\mathrm{dist}^{m}(x,\overline{PP_{1}})}+C
\left(\frac{\mathrm{dist}^{m-1}(x,\overline{PP_{1}})}{\varepsilon+\mathrm{dist}^{m}(x,\overline{PP_{1}})}+1\right)\|\varphi\|_{C^{2}(\partial{D})},\quad\,x\in{D}\setminus{D}_{1},
\end{align}
and if $|Q[\varphi]|\geq\,c^{*}$ for some universal constant $c^{*}>0$, then
\begin{equation}\label{lowerbound1m}
|\nabla{u}(x)|\geq
\frac{\rho_{n}^{m}(\varepsilon)\left|Q[\varphi]\right|}{C\varepsilon}, \qquad\,x\in\overline{PP_{1}},
\end{equation}
where $Q[\varphi]$
is defined by \eqref{def_Qphi}.
\end{theorem}

\section{Proof of Theorem \ref{thm2}}\label{extend}

Following the approach developed in the proof of Theorem \ref{thm1}, we construct an auxiliary function
$\tilde{u}\in{C}^{2}(\mathbb{R}^{n})$, such that $\tilde{u}=1$ on
$\partial{D}_{1}$, $\tilde{u}=0$ on
$\partial{D}$,
\begin{align}\label{utilde}
\tilde{u}(x)
=\bar{u}(x)+\frac{\sum_{i=1}^{n-1}A_{ni}(x)\partial_{x_{i}}(h_{1}-h)(x')}{4A_{nn}(x)}\left(\Big(\frac{2x_{n}-(\varepsilon+h_{1}(x')+h(x'))}{\varepsilon+h_{1}(x')-h(x')}\Big)^{2}-1\right),\quad\mbox{in}~~\Omega_{R_{0}},
\end{align}
and
\begin{equation}\label{nablau_tilde_outside}
\|\tilde{u}\|_{C^{2}(\Omega\setminus \Omega_{R_{0}})}\leq\,C.
\end{equation}
Using the assumptions on $h_{1}$ and $h$, \eqref{h1h0}--\eqref{h1h3}, a direct calculation still gives
\begin{equation}\label{nablau_tilde}
\frac{1}{C(\varepsilon+|x'|^{2})}\leq\left|\nabla\tilde{u}(x)\right|\leq\frac{C}{\varepsilon+|x'|^{2}},\quad~x\in\Omega_{R_{0}}.
\end{equation}
More importantly, thanks to the corrector term in \eqref{utilde}, we obtain the following bound
\begin{equation}\label{f_tilde}
\left|\partial_{i}(A_{ij}(x)\partial_{j}\tilde{u}(x))\right|\leq\frac{C}{\varepsilon+|x'|^{2}},\quad~x\in\Omega_{R_{0}},
\end{equation}
the same as \eqref{nabla2u_bar}. This is the point, which plays an important role in the proof of the following Proposition.

\begin{prop}\label{prop2}
Assume the above, let $V_{0}, V_{1}\in{H}^1(\Omega)$ be the
weak solution of \eqref{section3 equ_v0} and \eqref{section3 equ_v1}, respectively. Then
\begin{equation}\label{nabla_w_2}
\|\nabla(V_{1}-\tilde{u})\|_{L^{\infty}(\Omega)}\leq\,C,
\end{equation}
and
\begin{equation}\label{nabla_v1+v2_2}
|\nabla V_{0}(x)|\leq\,C\left(\frac{|x'|}{\varepsilon+|x'|^{2}}+1\right)\|\varphi\|_{C^{2}(\partial{D})},\quad\,x\in\Omega.
\end{equation}
Consequently,
\begin{equation}\label{nabla_v12}
\frac{1}{C(\varepsilon+|x'|^{2})}\leq|\nabla{V}_{1}(x)|\leq\frac{C}{\varepsilon+|x'|^{2}},\quad\,x\in\Omega_{R_{0}},
\end{equation}
and
$$
|\nabla{V}_{1}(x)|\leq\,C,\quad\quad\ \,x\in \Omega\setminus \Omega_{R_{0}}.
$$
\end{prop}

\begin{proof}[Proof of Proposition \ref{prop2}]

\noindent\textbf{STEP 1.} Proof of \eqref{nabla_w_2}.

Let
$$\widetilde{w}_{1}=V_{1}-\tilde{u}.$$
Similarly, instead of \eqref{w20}, we have
\begin{equation}\label{divergence w1}
\begin{cases}
-\displaystyle\partial_{i}(A_{ij}(x)\partial_{j} \widetilde{w}_{1})=\displaystyle\partial_{i}(A_{ij}(x)\partial_{j} \tilde{u})=:\tilde{f},\quad&\mbox{in }\Omega,\\
\widetilde{w}_{1}=0, &\mbox{on }\partial \Omega.
\end{cases}
\end{equation}
By the standard elliptic theory,
\begin{equation}\label{nabla_wtilde_out}
|\widetilde{w}_{1}|+\left|\nabla\widetilde{w}_{1}\right|\leq\,C,
\quad\mbox{in}~~ \Omega\setminus\Omega_{R_{0}}.
\end{equation}
On the other hand, by the maximum principle, we have
\begin{equation}\label{wtilde_bdd}
\|\widetilde{w}_{1}\|_{L^{\infty}(\Omega)}\leq\,C.
\end{equation}

\textbf{STEP 1.1. Boundedness of the energy.}
Multiplying the equation in (\ref{divergence w1}) by $\tilde{w}_{1}$, integrating by parts, using (\ref{ene01a}), (\ref{wtilde_bdd}) and \eqref{f_tilde}, we have
\begin{align*}
\lambda\int_{\Omega}|\nabla \widetilde{w}_{1}|^{2}dx
\leq\int_{\Omega}A_{ij}\partial_{i}\widetilde{w}_{1}\partial_{j}\widetilde{w}_{1}dx
=\,\int_{\Omega}\tilde{f}\,\widetilde{w}_{1}dx
\leq\,\|\widetilde{w}_{1}\|_{L^{\infty}(\Omega)}\left(\int_{\Omega_{R_{0}}}|\tilde{f}|dx+C\right)\leq\,C.
\end{align*}
So that
\begin{equation}\label{energy_w_1}
\int_{\Omega}\left|\nabla \widetilde{w}_{1}\right|^{2}dx\leq\,C.
\end{equation}

\textbf{STEP 1.2. Local energy estimates.}
Multiplying the equation in (\ref{divergence w1}) by $\eta^2\widetilde{w}_{1}$, where $\eta$ is the same cutoff function defined before, and integrating by parts, we deduce
\[\int_{\widehat{\Omega}_{s}(z^{\prime})}\displaystyle A_{ij}\partial_{j} \widetilde{w}_{1} \partial_{i}(\eta^2\widetilde{w}_{1})dx=\int_{\widehat{\Omega}_{s}(z^{\prime})}\displaystyle \tilde{f}\eta^2\widetilde{w}_{1}dx.\]
Then
\begin{equation*}
\begin{aligned}
\int_{\widehat{\Omega}_{s}(z^{\prime})}\displaystyle A_{ij}(\eta\partial_{j}\widetilde{w}_{1})(\eta \partial_{i}\widetilde{w}_{1})dx
&=\int_{\widehat{\Omega}_{s}(z^{\prime})}\bigg[-2\displaystyle A_{ij}(\eta\partial_{j}\widetilde{w}_{1})\widetilde{w}_{1}\partial_{i}\eta+\tilde{f}\eta^{2}\widetilde{w}_{1}\bigg]dx.
\end{aligned}
\end{equation*}
By \eqref{ene01a} and the Cauchy inequality,
\begin{eqnarray*}
\lambda\int_{\widehat{\Omega}_{s}(z^{\prime})}|\eta\nabla\widetilde{w}_{1}|^2dx\leq
\frac{\lambda}{4}\int_{\widehat{\Omega}_{s}(z^{\prime})}|\eta\nabla\widetilde{w}_{1}|^2dx+\frac{C}{(s-t)^{2}}\int_{\widehat{\Omega}_{s}(z^{\prime})}\widetilde{w}_{1}^{2}dx+C(s-t)^{2}\int_{\widehat{\Omega}_{s}(z^{\prime})}|\tilde{f}|^{2}dx.
\end{eqnarray*}
Thus
\begin{equation}\label{Ca-2}
\int_{\widehat{\Omega}_{t}(z^{\prime})}|\nabla \widetilde{w}_{1}|^2dx\leq\frac{C}{(s-t)^2}\int_{\widehat{\Omega}_{s}(z^{\prime})}|\widetilde{w}_{1}|^2dx+C(s-t)^{2}\int_{\widehat{\Omega}_{s}(z^{\prime})}|\tilde{f}|^2dx.
\end{equation}
Then using estimate \eqref{f_tilde}, instead of \eqref{nabla2u_bar}, we have
$$\int_{\widehat{\Omega}_{s}(z^{\prime})}|\tilde{f}|^2dx\leq
\begin{cases}
\frac{Cs^{n-1}}{\varepsilon},\quad&\mbox{if}~~|z'|\leq\sqrt{\varepsilon},\\
\frac{Cs^{n-1}}{|z'|^{2}},&\mbox{if}~~\sqrt{\varepsilon}<|z'|\leq\,R_{0}.
\end{cases}$$
Using the iteration argument, similar as step 1.2 in the proof of Proposition \ref{prop1}, we have $\widetilde{w}_{1}$ also satisfies \eqref{energy_w_inomega_z1}, that is,
\begin{equation*}
\int_{\widehat{\Omega}_{\delta}(z')}\left|\nabla{\widetilde{w}_{1}}\right|^{2}dx\leq
\begin{cases}C\varepsilon^{n},&\mbox{if}~~|z'|\leq\sqrt{\varepsilon},\\
C|z'|^{2n},&\mbox{if}~~\sqrt{\varepsilon}<|z'|\leq\,R_{0}.
\end{cases}
\end{equation*}
Thus, similar as step 1.3 in the proof of Proposition \ref{prop1},  \eqref{nabla_w_2} is established.

\noindent{\bf STEP 2.} Proof of \eqref{nabla_v1+v2_2}.

Using $\tilde{u}$ instead of $\bar{u}$, we define a function $\hat{u}_{2}\in C^{2}(\mathbb R^{n})$, such that $\hat{u}_{2}=0$ on $\partial{D}_{1}$, $\hat{u}_{2}=\varphi(x)-\varphi(0\,)$ on $\partial{D}$,
\begin{align}\label{uhat}
\hat{u}_{2}(x)
=(1-\tilde{u}(x))\Big(\varphi(x^{\prime},h(x^{\prime}))-\varphi(0)\Big),\quad\mbox{in}~~\Omega_{R_{0}},
\end{align}
and
\begin{equation}\label{nablau_hat_outside}
\|\hat{u}_{2}\|_{C^{2}(\mathbb{R}^{n}\setminus \Omega_{R_{0}})}\leq\,C.
\end{equation}
Denote
\begin{equation*}
\widetilde {w}_{0}:= V_{0}-\hat{u}_{2}.
\end{equation*}
Instead of \eqref{w}, we have
\begin{equation}\label{p-laplace w}
\begin{cases}
-\displaystyle\partial_{i}(A_{ij}(x)\partial_{j} \widetilde {w}_{0})=\displaystyle\partial_{i}(A_{ij}(x)\partial_{j} \hat{u}_{2})=:\hat{f},\quad&\mbox{in }\Omega,\\
\widetilde {w}_{0}=0, &\mbox{on }\partial \Omega.
\end{cases}
\end{equation}

\textbf{STEP 2.1.} The boundedness of the energy is the same as step 1.1. By a direct computation, we have
\begin{equation}\label{f_hat}
\left|\hat{f}\right|\leq\,C\Big(\frac{|x'|}{\varepsilon+|x'|^{2}}+1\Big)\|\varphi\|_{C^{2}(\partial{D})},\quad~x\in\Omega_{R_{0}}.
\end{equation}
and
$$\int_{\widehat{\Omega}_{s}(z^{\prime})}|\hat{f}|^2dx\leq\,C\|\varphi\|_{C^{1}(\partial{D})}^{2}
\begin{cases}
\frac{s^{n-1}}{\varepsilon},\quad&\mbox{if}~~|z'|\leq\sqrt{\varepsilon},\\
\frac{s^{n-1}}{|z'|^{2}},&\mbox{if}~~\sqrt{\varepsilon}<|z'|\leq\,R_{0},
\end{cases}$$
similar as step 1.2 in the proof of Proposition \ref{prop1}, $\widetilde {w}_{0}$ also satisfies \eqref{energy_w_inomega_z1}. The rest is the same. Proposition \ref{prop2} is established.
\end{proof}

\begin{proof}[Proof of Theorem \ref{thm2}]
Similarly as in the proof of Theorem \ref{thm1}, we decompose the solution $u$ of \eqref{equinftydiv} as
\begin{equation}\label{section3 decomposition_u}
u(x)=(C_{1}-\varphi(P\,))V_{1}(x)+V_{0}(x)+\varphi(P\,),\quad\mbox{in}~~D\setminus\overline{D}_{1}.
\end{equation}
Define
$$a_{11}:=-\int_{\partial{D}_{1}}A_{ij}(x)\,\partial_{j}{V}_{1}\,\nu_{i}.
$$
By integrating by parts,
\begin{align*}
0=\int_{\Omega}\partial_{i}(A_{ij}(x)\partial_{j}V_{1})\cdot\,V_{1}
=&-\int_{\Omega}A_{ij}(x)\partial_{j}V_{1}\partial_{i}V_{1}-\int_{\partial{D}_{1}}A_{ij}(x)\,\partial_{j}{V}_{1}\,\nu_{i}\cdot\,1\\
=&-\int_{\Omega}A_{ij}(x)\partial_{i}V_{1}\partial_{j}V_{1}+a_{11}.
\end{align*}
That is,
$$a_{11}=\int_{\Omega}A_{ij}(x)\partial_{i}V_{1}\partial_{j}V_{1}.$$
By the uniform elliptic condition \eqref{ene01a}, and \eqref{nabla_v12},
$$\frac{1}{C}\int_{\Omega}\frac{1}{(\varepsilon+|x'|^{2})^{2}}\leq\lambda\int_{\Omega}|\nabla\,V_{1}|^{2}\leq\,a_{11}\leq\Lambda\int_{\Omega}|\nabla\,V_{1}|^{2}\leq\,{C}\int_{\Omega}\frac{1}{(\varepsilon+|x'|^{2})^{2}}.$$
So that Lemma \ref{lem_bly} holds still. Then, combining with Proposition \ref{prop2}, the proof of Theorem \ref{thm2} is completed.
\end{proof}

\noindent{\bf{\large Acknowledgements.}} The authors would like to express their gratitude to Professor Jiguang Bao and YanYan Li's encouragement and very helpful discussions. The first author was partially supported by  NSFC (11571042) (11371060) (11631002), Fok Ying Tung Education Foundation (151003) and the Fundamental Research Funds for the Central Universities.


\end{document}